\documentclass[4paper,11pt]{amsart}

\usepackage{amsmath,amssymb,amsthm}
\usepackage[dvipdfm,pdfstartview={FitH  
-32768},bookmarks=false,colorlinks=false]{hyperref}\input xy
\xyoption{all}
\newtheorem{theorem}{Theorem}[section]

\newtheorem{cor}[theorem]{Corollary}
\newtheorem{lemm}[theorem]{Lemma}
\newtheorem{prop}[theorem]{Proposition}
\newtheorem{question}[theorem]{Question}
\theoremstyle{definition}
\newtheorem{defi}[theorem]{Definition}
\newtheorem{remk}[theorem]{Remark}

\newtheorem{exam}[theorem]{Example}
\newtheorem{observ}[theorem]{Observation}

\newcommand{\PP}{{\mathcal P}}

\newcommand{\QQ}{{\mathcal W}}

\newcommand{\bb}{\mbox{\boldmath $b$}}
\newcommand{\hh}{\mbox{\boldmath $f$}}
\newcommand{\aaa}{\mbox{\boldmath $a$}}
\newcommand{\rr}{\mbox{\boldmath $r$}}
\newcommand{\g}{\mbox{\boldmath $g$}}
\newcommand{\LL}{\mbox{\boldmath $\ell$}}

\newcommand{\rad}{\operatorname{rad}\nolimits}

\newcommand{\add}{\operatorname{add}\nolimits}

\newcommand{\gl}{\operatorname{gl.dim}\nolimits}
\newcommand{\pd}{\operatorname{pd}\nolimits}

\newcommand{\id}{\operatorname{id}\nolimits}

\newcommand{\Hom}{\operatorname{Hom}\nolimits}

\newcommand{\Ext}{\operatorname{Ext}\nolimits}

\renewcommand{\mod}{\operatorname{mod}\nolimits}

\def\Im{\mathop{\mathrm{Im}}\nolimits}
\def\Ker{\mathop{\mathrm{Ker}}\nolimits}

\def\Hom{\mathop{\mathrm{Hom}}\nolimits}
\def\End{\mathop{\mathrm{End}}\nolimits}

\def\Ext{\mathop{\mathrm{Ext}}\nolimits}

\def\mod{\mathop{\mathrm{mod}}\nolimits}

\def\add{\mathop{\mathrm{add}}\nolimits}

\begin{document}
\title[APR tilting modules and graded QPs]{APR tilting modules and graded quivers with potential}
\author{Yuya Mizuno}
%\date{\today}
\address{Graduate School of Mathematics\\ Nagoya University\\ Frocho\\ Chikusaku\\ Nagoya\\ 464-8602\\ Japan}
\email{yuya.mizuno@math.nagoya-u.ac.jp}

\begin{abstract}
We study quivers with relations of endomorphism algebras of APR tilting modules. 
It is known that the quivers are given by reflections if the original algebras have global dimension 1. We generalize this result for algebras with global dimension 2. In particular, we give an explicit description of the quivers with relations by mutations of quivers with potential (QPs). 
This result also provides a rich source of derived equivalence classes of algebras in a combinatorial way. 
As an application, we give a sufficient condition of QPs such that 
Derksen-Weyman-Zelevinsky's question \cite[Question 12.2]{DWZ} has a positive answer. 
\end{abstract}
\maketitle

\section{Introduction}
Derived categories have been one of the important tools in the study of many areas of mathematics. 
In the representation theory of algebras, tilting theory plays an essential role to control equivalences  of derived categories. More precisely, endomorphism algebras of tilting modules are derived equivalent to the original algebra. Therefore, the relationship between 
quivers with relations of endomorphism algebras and those of the original algebra is very important and the same has been investigated by various authors.

The first well-known result is a so-called BGP-reflection functor \cite{BGP}.
This is the origin of tilting theory, and now formulated in terms of APR tilting modules \cite{APR}. 
Let us recall an important property of APR tilting modules. 

\begin{theorem}\label{1}\cite{APR}
Let $KQ$ be a path algebra of a finite acyclic quiver $Q$ and $T_k$ 
be the APR tilting $KQ$-module associated with a source $k\in Q_0$.
Then we have an algebra isomorphism 
$$\End_{KQ}(T_k)\cong K({\mu}_kQ),$$ 
where $\mu_k$ is a reflection at $k$.
\end{theorem}

Thus, in the case of path algebras, the quivers of endomorphism algebras of APR tilting modules are completely determined by combinatorial methods.

One of the main purposes of this paper is to generalize the above result for a more general class of algebras. Since path algebras have global dimension at most 1, it is natural to consider  algebras with global dimension at most 2. 
In order to deal with these algebras, we use mutations of quivers with potential (QPs).  

The notion of mutation was introduced by Fomin-Zelevinsky \cite{FZ}, which is
an important ingredient in cluster algebras. After that, Derksen-Weyman-Zelevinsky extend  mutations for QPs. 
It has recently been discovered that mutations of QPs have a close connection with tilting theory, for example \cite{IR,BIRSm,KY}. 
In this paper, we provide another type of connection between mutations and tilting theory.
 
We consider the following steps for an algebra $\Lambda$.

\begin{itemize}
\item[1.] Define the associated graded QP $(Q_\Lambda,W_\Lambda,C_\Lambda)$.
\item[2.] Apply left mutation ${\mu}_k^L$ to $(Q_\Lambda,W_\Lambda,C_\Lambda)$.
\item[3.] Take the truncated Jacobian algebras $\PP({\mu}_k^L(Q_\Lambda,W_\Lambda,C_\Lambda))$.
\end{itemize}
Then we can give a generalization of Theorem \ref{1} as follows.

\begin{theorem}\label{2}$($Theorem \ref{main}$)$
Let $\Lambda$ be a finite dimensional algebra with $\gl\Lambda\leq2$ and $T_k$ be the APR tilting $\Lambda$-module associated with a source $k$.
Then we have an algebra isomorphism
$$\End_\Lambda(T_k)\cong {\mathcal P}({\mu}_k^L(Q_\Lambda,W_\Lambda,C_\Lambda)).$$ 
\end{theorem}

We give two remarks about the theorem. 
First, if $\gl\Lambda= 1$, then ${\mathcal P}({\mu}_k^L(Q_\Lambda,W_\Lambda,C_\Lambda))$ is isomorphic to $K({\mu}_kQ)$. 
Thus, Theorem \ref{2} implies Theorem \ref{1}. 
Secondly, the condition $\gl\Lambda\leq2$ is not necessary. In fact, it is enough to assume that the associated projective module has the injective dimension at most 2. 
%We note that APR tilting modules and graded mutations have entirely different origins and backgrounds.

Our second aim is to formulate the above result in terms of algebraic cuts and to study the 
relationships between graded QPs and tilting theory. 
We will show that, from a given QP with an algebraic cut, the new QP obtained by mutation at a strict source also has an algebraic cut (Proposition \ref{algebraic}). 
It implies that successive mutations give APR-tilted (co-tiled) algebras  
and, consequently, provide a family of derived equivalence classes of algebras. 
It also allows one to treat several representation theoretical problems by combinatorial methods. 
Moreover, these treatments induce an analog of the Keller-Yang's approach \cite{KY}, where instead of Jacobian algebras $\PP(Q,W)$ and $\PP(\mu_k(Q,W))$, 
the authors investigated the Ginzburg dg algebras $\Gamma(Q,W)$ and $\Gamma(\mu_k(Q,W))$ and showed the derived equivalences between them. 
In our cases, we obtain the derived equivalences between the truncated Jacobian algebras $\PP(Q,W,C)$ and $\PP(\mu_k^L(Q,W,C))$ (Proposition \ref{main3}). 

Finally, we give an application to the question posed in \cite{DWZ}.
\begin{question}\cite[Question 12.2]{DWZ}
Is the isomorphism class of $\PP(\mu_k(Q,W))$ determined by the isomorphism
class of $\PP(Q,W)$?

\end{question}
 
We give a sufficient condition for QPs such that the question has a positive answer. (Proposition \ref{app1}). 

\subsection*{Notation and convention}
Let $K$ be an algebraically closed field and put $D:=\Hom_K(-,K)$. 
Let $\Lambda$ be a finite dimensional algebra. 
We denote by $J_\Lambda$ the Jacobson radical of $\Lambda$. 
All modules are left modules. 
We denote by $\mod\Lambda$ the category of finitely generated $\Lambda$-modules and by add$M$ the subcategory consisting of direct summands of finite direct sums of $M$. 
The composition $fg$ of morphisms means first $f$, then $g$.
We denote the set of vertices by $Q_0$ and the set of arrows by $Q_1$ of a quiver $Q$. 
We denote by $a:s(a)\to e(a)$ the start and end vertices of an arrow or path $a$. 
We denote by $\pd X$ and $\id X$, respectively, the projective dimension and injective dimension of a  module $X$.

\section{Background}
In this section, we give a summary of the definitions and results we will use in the next sections. 
See references for more detailed arguments and precise definitions. 
%%%%%%%%%%%%%%%%%
\subsection{Quivers with potential}We review the notions initiated in \cite{DWZ}.

$\bullet$ Let $Q$ be a finite connected quiver (possibly with loops and 2-cycles).
We denote by $KQ_i$ the $K$-vector space with basis consisting of paths of length $i$ in $Q$, and by
$KQ_{i,cyc}$ the subspace of $KQ_i$ spanned by all cycles.
We denote the \emph{complete path algebra} by
$$\widehat{KQ}=\prod_{i\ge0}KQ_i.$$ 
A \emph{quiver with potential} (QP) is a pair $(Q,W)$ consisting of a quiver $Q$ and 
an element $W\in\prod_{i\ge2}KQ_{i,{\rm cyc}}$, called a \emph{potential}.
For each arrow $a$ in $Q$, the \emph{cyclic derivative} $\partial_a:\widehat{KQ}_{cyc} \to \widehat{KQ}$  is defined by the continuous linear map 
satisfying $\partial_a(a_1\cdots a_d)=\sum_{a_i=a}a_{i+1}\cdots a_d a_1\cdots a_{i-1}$ for a cycle $a_1\cdots a_d$.
For a QP $(Q,W)$, we define the \emph{Jacobian algebra} by
\[\PP(Q,W)=\widehat{KQ}/{\mathcal J}(W),\]
where ${\mathcal J}(W)=\overline{\langle \partial_a W \mid a \in Q_1 \rangle}$
is the closure of the ideal generated by $\partial_aW$ with respect to the $J_{\widehat{KQ}}$-adic topology.

$\bullet$ 
Two potentials $W$ and $W'$ are called \emph{cyclically equivalent} if $W-W'$ lies in 
the closure of the span of all elements of the form $a_1\cdots a_d-a_2\cdots a_da_1$, where $a_1\cdots a_d$ is a cyclic path.

$\bullet$   
A QP $(Q,W)$ is called \emph{trivial} if $W$ is a linear combination
of cycles of length 2 and $\PP(Q,W)$ is isomorphic to the semisimple algebra $\widehat{KQ_0}$.
It is called \emph{reduced} if $W\in\prod_{i\ge3}KQ_{i,{\rm cyc}}.$

$\bullet$ For two QPs $(Q',W')$ and $(Q'',W'')$ such that $Q_0'=Q_0''$, we define a new QP $(Q,W)$ as a direct sum $(Q',W')\oplus(Q'',W'')$, where 
$Q_0=Q_0'(=Q_0'')$, $Q_1=Q_1'\coprod Q_1''$ and $W=W'+W''$.

%\begin{defi}
$\bullet$ For each vertex $k$ in $Q$ not lying on a loop (but possibly lying on a 2-cycle), we define a new QP $\widetilde{\mu}_k(Q,W):=(Q',W')$ as follows. %Replacing $W$ by a cyclically equivalent potential, we can assume that no cycles in $W$ start and end at $k$.

\begin{enumerate}

\item$Q'$ is a quiver obtained from $Q$ by the following changes.

$\bullet$ Replace each arrow $a:k\to v$ in $Q$ by a new arrow $a^*:v \to k$.

$\bullet$ Replace each arrow $b:u\to k$ in $Q$ by a new arrow $b^*:k \to u$.

$\bullet$ For each pair of arrows $u\overset{b}{\to} k\overset{a}{\to}  v$, 
add a new arrow $[ba] : u \to v$.
\item$W' = [W] + \Delta$ is defined as follows.

$\bullet$  $[W]$ is obtained from the potential $W$ by replacing all compositions $ba$ by the new arrows $[ba]$ for each pair of arrows $u\overset{b}{\to} k\overset{a}{\to}  v$. 

$\bullet$  $\Delta={\displaystyle\sum_{\begin{smallmatrix}a,b\in Q_1\\
e(b)=k=s(a)\end{smallmatrix}}}[ba]a^*b^*$. 
\end{enumerate}
%\end{defi}

\subsection{Truncated Jacobian algebras}

\begin{defi}\label{truncated}\cite{HI}
Let $(Q,W)$ be a QP.
A subset $C \subset Q_1$ is called a \emph{cut} if each cycle appearing $W$ contains exactly one arrow of $C$. Then we define the \emph{truncated Jacobian algebra} by
$$\PP(Q,W,C) := {\mathcal P}(Q,W)/ \overline{\langle C \rangle} = \widehat{ KQ_C}/\overline{ {\langle\partial_cW\ |\ c \in C  \rangle} },$$
where $Q_C$ is the subquiver of $Q$ with vertex set $Q_0$ and arrow set $Q_1\setminus C$.
\end{defi}

Here, we recall some elementary notions. 
A \emph{basic} element of $\widehat{KQ}$ is a formal linear sum of paths with a
common start and a common end. 
Basic elements are called \emph{relations} if each element is a path of length at least 2.
A set of relations $R$ is called \emph{minimal} if $r\notin \overline{ {\langle R \setminus\{r\} \rangle}}$ for any $r\in R$. 
Then, as introduced by Keller, we can naturally define a QP with a cut from a given algebra. The reason we deal with completion is noted in \cite[Section 7]{BIKR}. 

\begin{defi}\cite{K}\label{construct QP}
Let $Q$ be a finite connected quiver and 
$\Lambda = \widehat{KQ}/\overline{{\langle R\rangle}}$ be a finite dimensional algebra with a minimal set  $R$ of relations.

Then we define a QP $(Q_\Lambda,W_\Lambda)$ as follows.
\begin{itemize}
\item[(1)]  $({Q_\Lambda})_0 = Q_0$.
\item[(2)]  $({Q_\Lambda})_1 = Q_1 \coprod C_\Lambda,$ where $C_\Lambda:=\{ \rho_r:e(r)\to s(r)\ |\ r \in R \}$. 
\item[(3)]  $W_\Lambda = {\displaystyle\sum_{\begin{smallmatrix}r \in R\end{smallmatrix}}}\rho_r r$.
\end{itemize}
\end{defi}

%(Note that $R$ is not unique.)

Then the set $C_\Lambda$ gives a cut of $(Q_\Lambda,W_\Lambda)$.

%%%%%%%%%%%%%%%%%
\subsection{APR tilting modules}

We call a $\Lambda$-module $T$ \emph{tilting module} if 
(i)  $\pd T\leq 1$, (ii)  $\Ext_\Lambda^1(T,T)=0,$ and (iii) there exists a short exact sequence 
$0\rightarrow \Lambda \rightarrow T_0 \rightarrow T_1 \rightarrow0$ with $T_0,T_1$ in $\add T$.
 
We call a $\Lambda$-module $T'$ \emph{co-tilting module} if $DT'$ is a tilting $\Lambda^{op}$-module. Then we introduce one of the most fundamental classes of tilting modules as follows.

\begin{defi}\label{tilt}
Let $\Lambda$ be a basic finite dimensional algebra and $P_k$ be a simple projective non-injective $\Lambda$-module associated with a source $k$ of the quiver of $\Lambda$. 
Then, the $\Lambda$-module $T_k:=\tau^-P_k\oplus\Lambda/P_k$ is called an \emph{APR tilting module}, where $\tau^-$ denotes the inverse of the Auslander-Reiten translation. 
Dually, we define APR co-tilting module. %$D\Lambda/{I}\oplus \tau I$ for a simple injective non-projetive module $I$. 
\end{defi} 

We denote by $\nu_k^L$ a \emph{left tilting mutation} \cite{HU,RS,U} for a source and non-sink $k$, which is given by a minimal left $\add(\Lambda/P_k)$-approximation of $P_k$ and its cokernel. Then, we have $\nu_k^L(\Lambda)=\tau^-P_k\oplus\Lambda/P_k$. Thus, APR tilting modules are induced by tilting mutation. Dually, we denote by $\nu_k^R$ a \emph{right co-tilting mutation} for a sink and non-source $k$. 
%and $\nu_k^R(\Lambda)=\tau I_k\oplus D\Lambda/I_k$.  

%%%%%%%%%%%%%%%%%%%%%%%%%%%%%%%%%%%%%%%%%%%%%%%%%%%%%%%%%%%%%%%%%%%%%%%%%%%%%%%%%%%%%%%%%%%%%%
\section{Main theorem}
%%%%%%%%%%%%%%%%%%%%%%%%%%%%%%%%%%%%%%%%%%%%%%%%%%%%%%%%%%%%%%%%%%%%%%%%%%%%%%%%%%%%%%%%%%%%%%
\subsection{Main result}

Let $Q$ be a finite connected quiver %(possibly with loops)
and $\Lambda = \widehat{KQ}/\overline{{\langle R\rangle}}$ be a finite dimensional algebra with a minimal set of relations. Assume that $P_k$ is the simple projective non-injective $\Lambda$-module associated with a source $k\in Q_0$. Our aim is to determine the quiver and the set of relations of $\End_\Lambda(T_k)$.

Consider the associated QP $(Q_\Lambda,W_\Lambda,C_\Lambda)$ of $\Lambda$ and we put 
$\widetilde{\mu}_k(Q_\Lambda,W_\Lambda,C_\Lambda)=(Q',W',C')$, 
where $W'$ and $C'$ are given by 
$$W'=\ [{\displaystyle\sum_{\begin{smallmatrix}r \in R\end{smallmatrix}}}  \rho_rr]+{\displaystyle\sum_{\begin{smallmatrix}a\in Q_1,r\in R\\
s(a)=k=s(r)\end{smallmatrix}}}[\rho_ra]a^*\rho_r^*,$$ 
$$C'=\{\ \rho_r\ |\ r\in R,\ s(r)\neq k \} \coprod\{\ [\rho_ra]\ |\ a\in Q_1,\ r\in R,\ s(a)=k=s(r)\}.$$
It is easy to check that $C'$ is a cut of $(Q',W')$. 

Then, we can achieve our aim as follows.
our main result of this paper is the following theorem. 
\begin{theorem}\label{main}
Let $\Lambda=\widehat{KQ}/\overline{{\langle R\rangle}}$ be a finite dimensional algebra with a minimal set $R$ of relations and  
%Assume that $P_k$ is the simple projective non-injective $\Lambda$-module associated with a source $k\in Q$.  
$T_k:=\tau^-P_k\oplus \Lambda/P_k$ be the APR tilting module.
Then if $\id P_k\leq 2$, we have an algebra isomorphism $$\End_\Lambda(T_k)\cong {\mathcal P}(\widetilde{\mu}_k(Q_\Lambda,W_\Lambda,C_\Lambda)).$$ 
\end{theorem}
Note that the assumption $\id P_k\leq 2$ is automatic if $\gl\Lambda=2$.
%Thus our theorem give a generalization of Theorem \ref{1} from $\gl\Lambda=1$ to $\gl\Lambda=2$.

Here, we will explain the choice of $C'$. In fact, $C'$ is naturally obtained by using graded mutations. For this purpose, we recall graded QPs. We follow \cite{AO}. 

\subsection*{Graded quivers with potential}
\begin{defi}
Let $(Q,W)$ be a QP. If we define a map $d:Q_1\to\mathbb{Z}$, then the degree map on arrows induces a degree map on paths. 
We call a QP $(Q,W,d)$ $\mathbb{Z}$-\emph{graded QP} if each arrow $a\in Q_1$ has a degree $d(a)\in\mathbb{Z}$, and  \emph{homogeneous of degree} $l$ if each term in $W$ has a degree $l$. 
In this case, $(Q,W,d)$ naturally induces a grading on $\widehat{KQ}=\prod_{i\ge0}(\widehat{KQ})_i$ and $\PP(Q,W)=\prod_{i\ge0}\PP(Q,W)_i$. 

$\bullet$ Two $\mathbb{Z}$-graded QPs $(Q,W,d)$ and $(Q',W',d')$ are called \emph{graded right-equivalent} if $Q_0=Q_0'$ and 
there exists a degree preserving isomorphism
$\phi: (\widehat{KQ},d)\to(\widehat{KQ'},d')$ such that $\phi|_{\widehat{KQ_0}}={\rm id}$ and $\phi(W)$ is cyclically equivalent to $W'$. 
 
$\bullet$ 
Let QP $(Q,W,d)$ be a $\mathbb{Z}$-graded QP of degree $l$.
For each vertex $k$ in $Q$ not lying on a loop (but possibly lying on a 2-cycle), we define a new graded QP $\widetilde{\mu}_k^L(Q,W,d):=(Q',W',d')$ as follows.

\begin{enumerate}
\def\labelenumi{(\theenumi)}
\item $(Q',W')=\widetilde{\mu}_k(Q,W).$
\item 
The new degree $d'$ is defined as follows:

$\bullet$ $d'(a)= d(a)$ for each arrow $a \in Q\cap Q'$.

$\bullet$ $d'(a^*) = -d(a)$ for each arrow $a:k\to v$ in $Q$.

$\bullet$ $d'(b^*) = -d(b) + l$ for each arrow $b:u\to k$ in $Q$.

$\bullet$ $d'([ba]) = d(a) + d(b)$ for each pair of arrows $u\overset{b}{\to} k\overset{a}{\to}  v$ in $Q$.
\end{enumerate}

In particular, $\widetilde{\mu}_k^L(Q,W,d)$ also has a potential of degree $l$. 
Similarly, we can define $\widetilde{\mu}_k^R$ at $k$. 
In this case, we define $d'(b^*) = -d(b)$ for each arrow $b:u\to k$ in $Q$ and 
$d'(a^*) = -d(a) + l$ for each arrow $a:k\to v$ in $Q$. 
\end{defi}

If $(Q,W)$ has a cut $C$, we can identify it with a $\mathbb{Z}$-graded QP of degree 1 
associating a grading on $Q$ by
\[
d_C(a) = \begin{cases} 1 & a \in C \\
0 & a \not \in C.
\end{cases}
\]

We denote by $(Q,W,C)$ the graded QP of degree 1 with this grading. 
Then we can interpret the truncated Jacobian algebra as the degree 0 part of $\PP(Q,W)$. 
If all arrows of $\widetilde{\mu}_k^L(Q,W,C)$ have degree 0 or 1, then degree 1 arrows give a cut of $\widetilde{\mu}_k(Q,W)$ since $\widetilde{\mu}_k^L(Q,W,C)$ is homogeneous of degree $1$. 
Therefore a cut of $\widetilde{\mu}_k(Q_\Lambda,W_\Lambda)$ is naturally induced as 
degree 1 arrows of left mutation $\widetilde{\mu}_k^L(Q_\Lambda,W_\Lambda,C_\Lambda)$ and the above $C'$ is 
obtained in this way. 
Furthermore, as in \cite[Theorem 6.4]{AO} and \cite[Theorem 4.6]{DWZ}, 
we consider the following lemma.

\begin{lemm}
For a graded QP $(Q,W,C)$, the graded splitting theorem holds. Namely 
$(Q,W,C)$ is graded right-equivalent to the direct sum 
$$(Q_{red},W_{red},C_{red})\oplus(Q_{tri},W_{tri},C_{tri}),$$ 
where $(Q_{red},W_{red},C_{red})$ is reduced and $(Q_{tri},W_{tri},C_{tri})$ is trivial, both unique up to graded right-equivalence.
\end{lemm}

\begin{proof}
Since each arrow has degree 0 or 1 and $W\in\prod_{i\ge2}KQ_{i,{\rm cyc}}$, there exists no potential consisting of one arrow. 
Therefore, a path of length 2 in $W$ is of the form $ab$, where $a$ is a degree 1 arrow and $b$ is a degree 0 arrow, and hence they are distinct arrows. 
Then, by the similar argument of \cite[Proposition 4.4]{DWZ}, the trivial QP is graded right-equivalent to a graded QP $(Q',W',C')$ such that $Q'$ consists of $2n$ distinct arrows and $W'=\sum_{i=1}^na_ib_i$ and $C'=\{a_i\ |\ i=1,\ldots,n \}$. 
Then the arguments of \cite[Theorem 4.6]{DWZ} still work and, as pointed out in \cite[Theorem 6.4, Proposition 6.5]{AO}, the right-equivalence given in the proof of \cite[Lemma 4.8]{DWZ} is a degree preserving isomorphism and  the reduction process does not change the homogeneity of the potential. 
\end{proof} 

We call $(Q_{red},W_{red},C_{red})$ a \emph{reduced part} of $(Q,W,C)$. 
%As a result, we can define a reduced part of $\widetilde{\mu}_k^L(Q_\Lambda,W_\Lambda,C_\Lambda)$. 
We denote a reduced part of $\widetilde{\mu}_k^L(Q_\Lambda,W_\Lambda,C_\Lambda)$ by ${\mu}_k^L(Q_\Lambda,W_\Lambda,C_\Lambda)$ and call it a \emph{left mutation}. 
Similarly, we denote a \emph{right mutation} by ${\mu}_k^R$. 
%((Remark that potentials of a path of length 1 do not appear. )
%Thus degree 1 arrows give a cut of $\mu_k^L(Q_\Lambda,W_\Lambda,C_\Lambda)$.  
%since each arrow has degree 0 or 1 and the potential is homogeneous of degree $1$. 
%Therefore degree 1 arrows give a cut of $\mu_k(Q_\Lambda,W_\Lambda)$ 

Because we have $\PP(\widetilde{\mu}_k^L(Q_\Lambda,W_\Lambda,C_\Lambda))\cong\PP({\mu}_k^L(Q_\Lambda,W_\Lambda,C_\Lambda))$, 
we can rewrite Theorem \ref{main} that we have an algebra isomorphism $$\End_\Lambda(T_k)\cong\PP(\mu_k^L(Q_\Lambda,W_\Lambda,C_\Lambda)).$$ 

%(Note that even though $k$ is source of $Q$, it may lie on a 2-cycle of $Q_\Lambda$ and $[\rho_ra]$ may be a loop (see Example \ref{exam1} (3)).)

\subsection{Examples}
In this subsection, we explain the theorem by some examples. We keep the assumption of Theorem \ref{main}.  
\begin{exam}
If $\gl\Lambda=1$, then we have $\Lambda=KQ$. Hence, we have 
$${\mathcal P}({\mu}_k^L(Q_\Lambda,W_\Lambda,C_\Lambda))
={\mathcal P}({\mu}_k^L(Q,0,\emptyset))= K({\mu}_kQ).$$ 
Thus, the mutation procedure is just reversing arrows having $k$. 
Therefore, the above theorem gives the classical result Theorem \ref{1}. 
\end{exam}

\begin{exam}\label{exam1}
Let $\Lambda=\widehat{KQ}/\overline{{\langle R\rangle}}$ be a finite dimensional algebra given by the following quiver with a relation.
$$
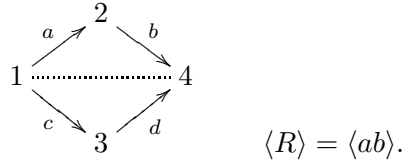
\xymatrix@C20pt@R10pt{ &2 \ar[rd]^{b}     & \\
1\ar[ru]^{a} \ar[rd]_{c}   &   &4  \ar@{.}[ll]&\\
 &3     \ar[ru]_d & &{\langle R\rangle}=\langle ab\rangle.}\ \ \ \ \ \ 
$$

Then we consider the APR tilting module $T_1:=\tau^-P_1\oplus \Lambda/P_1$ and calculate $Q'$ and $R'$  
satisfying $\widehat{KQ'}/\overline{{\langle R'\rangle}}\cong \End_\Lambda(T_1)$ by the following steps. Here, we denote degree 1 arrows by dotted arrows.

$$
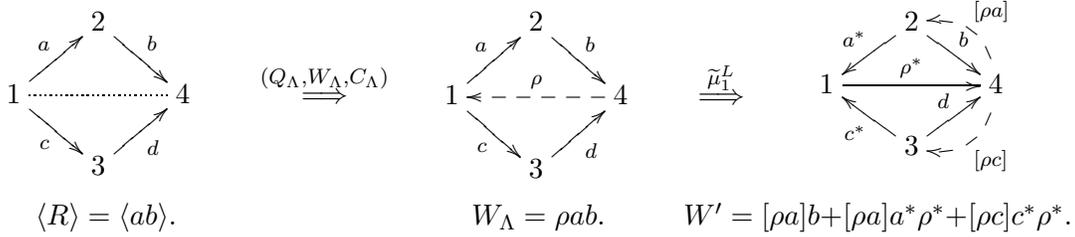
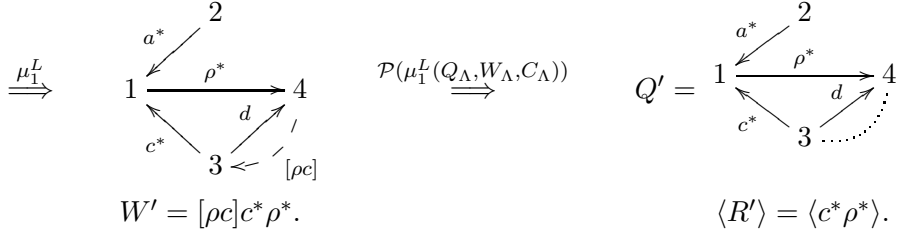
\xymatrix@C20pt@R10pt{ &2 \ar[rd]^{b}     & \\
1\ar[ru]^{a} \ar[rd]_{c}   &   &4  \ar@{.}[ll]&\overset{(Q_\Lambda,W_\Lambda,C_\Lambda)}{\Longrightarrow}\\
 &3     \ar[ru]_d & }\ \ \ \ 
 \xymatrix@C20pt@R10pt{ &2 \ar[rd]^{b}     & \\
1\ar[ru]^{a} \ar[rd]_{c}   &   &4  \ar@{-->}[ll]_{\rho}&\overset{\widetilde{\mu}_1^L}{\Longrightarrow}\\
 &3     \ar[ru]_d & }\ \ \ \ \ \ 
\xymatrix@C20pt@R10pt{ &2 \ar[ld]_{a^*}  \ar[rd]^{b}   \\
1  \ar[rr]^{\rho^*} &   &4 \ar@/_4mm/@{-->}[lu]_{[\rho a]}  \ar@/^4mm/@{-->}[ld]^{[\rho c]} \\ 
&3  \ar[ru]^{d}   \ar[lu]^{c^*}  }$$
$ \ \ \ \ \  \ \ \ \ \  \  \ \ \ {\langle R\rangle}=\langle ab\rangle . \ \ \ \ \  \ \  \ \ \ \ \ \ \ \  \ \ \ \ \ \ \ \ \  \ \ \ \ \ \ \ \ \  \ \ \ \ \ \ \ \ \  \ \ \ W_\Lambda= \rho ab.   \ \ \ \ \ \ \ \ \ \ \ \ W'= [\rho a]b +  [\rho a]a^*\rho^*+[\rho c]c^*\rho^*.$ 
 
$$\xymatrix@C20pt@R10pt{ &&2 \ar[ld]_{a^*}   & \\
\overset{{\mu}_1^L}{\Longrightarrow}&1  \ar[rr]^{\rho^*} &   &4  \ar@/^4mm/@{-->}[ld]^{[\rho c]}&\overset{\PP({\mu}_1^L(Q_\Lambda,W_\Lambda,C_\Lambda))}{\Longrightarrow}\ \ \ \ \ \ Q'= \\
 &&3  \ar[ru]^{d}   \ar[lu]^{c^*} & }
 \xymatrix@C20pt@R10pt{ &2 \ar[ld]_{a^*}     & \\
  1  \ar[rr]^{\rho^*} &   &4   \ar@/^4mm/@{.}[ld]^{}\\
 &3  \ar[ru]^{d}   \ar[lu]^{c^*} & }$$
$ \ \ \ \ \  \ \ \ \ \ \ \ \ \ \  \ \ \ \ \    \  \ \ \ W'= [\rho c]c^*\rho^*. \ \ \ \ \ \  \ \ \  \ \ \ \ \ \  \ \ \ \ \ \ \ \ \ \  \ \ \ \ \ \ \ \ \  \ \ \ \ \ \  \ \ \ {\langle R'\rangle}=\langle c^*\rho^* \rangle.$\\

Below we give more examples. 
From the left-hand side algebra $\Lambda$, we obtain the quiver and the set of relations of  $\End_\Lambda(T_1)$, which is given by right-hand side in the following diagrams. 
\begin{enumerate}

\item In this example, $Q_\Lambda$ coincide with the above quiver, but the new quiver is  different.
$$
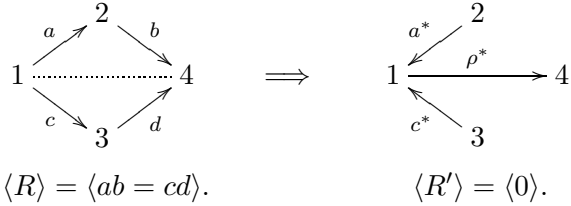
\xymatrix@C20pt@R10pt{ &2 \ar[rd]^{b}     & \\
1\ar[ru]^{a} \ar[rd]_{c}   &   &4  \ar@{.}[ll]&\overset{}{\Longrightarrow}\\
 &3     \ar[ru]_d & }\ \ \ \ \ \ 
\xymatrix@C20pt@R10pt{ &2 \ar[ld]_{a^*}     & \\
1  \ar[rr]^{\rho^*} &   &4  \\
 &3     \ar[lu]^{c^*} & }$$ 
$$ \ \ \ \ \ \ \ \ \ \ \ \ \ \ \ {\langle R\rangle}=\langle ab=cd\rangle. \   \ \ \ \ \ \ \ \ \ \ \  \ \ \ \ \ \ \ \ \ {\langle R'\rangle}=\langle 0 \rangle.\ \ \ \ \  \ \ \ \ \ \ \ \ \ \ \ \ \ $$

\item We do not need to assume that $\gl\Lambda\leq2$.

$
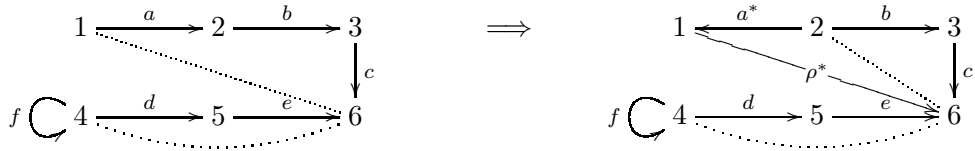
\xymatrix@C40pt@R20pt{
1\ar[r]^{a}   &2     \ar[r]^{b} &3 \ar[d]^c &\overset{}{\Longrightarrow}\\
4\ar@(ul,dl)_f\ar[r]^d    &5     \ar[r]^e &6\ar@{.}[llu] \ar@/^4mm/@{.}[ll]} \ \ \ \ \ \ \ \ \ \xymatrix@C40pt@R20pt{
1 \ar[rrd]|{\rho^*}&2  \ar[l]_{a^*}   \ar@{.}[rd]  \ar[r]^{b} &3 \ar[d]^c &\\
4\ar@(ul,dl)_f\ar[r]^d    &5     \ar[r]^e &6 \ar@/^4mm/@{.}[ll]}$

$$ \ \ \ \ \ \ \ \ \ \  \ \ \ \ \ \ \ \ \ \ \ \  {\langle R\rangle}=\langle abc,de,f^3 \rangle. \ \ \ \ \ \  \ \ \ \ \ \ \ \ \ \ \ \ \ \ \ \ \ \ \ \ \ \ \ \ \ \ \ \ {\langle R'\rangle}=\langle 
a^*\rho^*+bc,de,f^3 \rangle.$$

\item We do not need to assume that $Q_\Lambda$ has no multiple arrows.

$
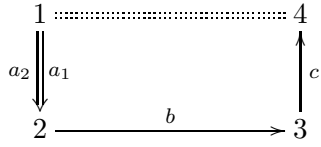
\xymatrix@C40pt@R30pt{
1\ar@{=>}[d]_{a_2}^{a_1}  &  &4 \ar@{:}[ll]&\overset{}{\Longrightarrow}  \\
2 \ar[rr]^b   & & 3 \ar[u]_{c} &  }$\ \ \ \ \ \ \  \ \
$\xymatrix@C40pt@R30pt{
1  \ar@{=>}[rr]_{\rho_1^*}^{\rho_2^*} &  &4\ar@{:}[lld]  \ar@/_4mm/@{.}[lld]   \\
2 \ar@/^6mm/@{.}[rru]\ar[rr]^b \ar@{=>}[u]_{a_1^* }^{a_2^*}  & & 3 \ar[u]^c}$
$$\ \ \ \ \ \ \ \ \ \ \ \ \  \ \ \ \ \ \ \ \ \ \ \ \ \ \ \ {\langle R\rangle}=\langle a_1bc,a_2bc \rangle. \ \ \ \ \ \ \ \ \ \ \ \ \ \ \ \ \ \ \ \ \ \ \ \ \ \ \ \ \ \  \ \ \ \ \  \ {\langle R'\rangle}=\langle a_1^*\rho_1^*+bc,a_2^*\rho_2^*+bc,a_1^*\rho_2^*,a_2^*\rho_1^*   \rangle.\ \ \ \ \ \  \ \ \ \ \ \ \ $$

\item A source vertex may lie on a 2-cycle of $Q_\Lambda$. 

$
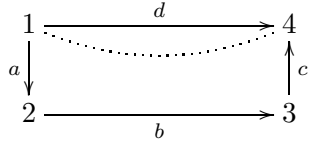
\xymatrix@C40pt@R20pt{
{1} \ar[rr]^{d} \ar[d]_{a}\ar@/_4mm/@{.}[rr] & &4 &\overset{}{\Longrightarrow}   \\
2 \ar[rr]_{b} &  &3  \ar[u]_{c }  } \ \ \ \ \ \ \ \ \ 
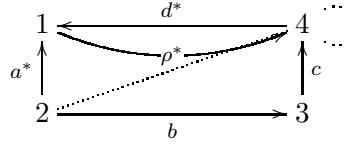
\xymatrix@C40pt@R20pt{
{1}  \ar@/_4mm/[rr]|{\rho^*} & &4  \ar@(ur,dr)@{.}[] \ar[ll]_{d^*}  \\
2 \ar[u]^{a^*} \ar@{.}[rru] \ar[rr]_{b} &  &3  \ar[u]_{c }  }$
$$  \ \ \ \ \ \ \ \ \ \ \  \ \ \ \ \ \ \ \ \ \ \ \ \ {\langle R\rangle}=\langle abc\rangle.   \ \ \  \ \ \ \ \ \ \  \ \ \ \ \ \  \ \ \ \ \ \ \ \ \ \  \ \ \ \ \ \ \ \ \  \ \ \ {\langle R'\rangle}=\langle a^*\rho^*+ bc,d^*\rho^*\rangle.\ \ \ \ \ \  \ \ \ \ \ \ \ $$

\end{enumerate}
As the examples show, we interpret degree 1 arrows as relations. 
We remark that these considerations also appear in the study of cluster tilted algebras. 
\end{exam}

\subsection{Proof of main theorem}

To give a proof of the theorem, we recall the following notation and result given by \cite{BIRSm}.  

Let $Q$ be a finite quiver. 
For $a\in Q_1$, define a \emph{right derivative} $\partial^r_a: J_{\widehat{KQ}}\to\widehat{KQ}$
and a \emph{left derivative} $\partial^l_a: J_{\widehat{KQ}}\to\widehat{KQ}$ by 
\begin{eqnarray*}
&\partial^r_a(a_1a_2\cdots a_{m-1}a_m)=
\left\{\begin{array}{cc}
a_1a_2\cdots a_{m-1}&\mbox{ if }\ a_m=a,\\
0&\mbox{ otherwise,}
\end{array}\right.&\\
&\partial^l_a(a_1a_2\cdots a_{m-1}a_m)=
\left\{\begin{array}{cc}
a_2\cdots a_{m-1}a_m&\mbox{ if }\ a_1=a,\\
0&\mbox{ otherwise,}
\end{array}\right.&
\end{eqnarray*}
and extend to $J_{\widehat{KQ}}$ linearly and continuously.
For the sake of simplicity of presentation, 
we write $Q(i,j):=\{a\in Q_1\ |\ s(a)=i,\ e(a)=j\}$ for a quiver $Q$ and $R(i,j):=\{ r\in R\ |\ s(r)=i,\ e(r)=j\}$ 
for a set of basic elements $R$. 

\begin{theorem}\label{birs2}\cite[Proposition 3.1, 3.3]{BIRSm}
Let $Q$ be a finite connected quiver and $\Gamma$ be a basic finite dimensional algebra.
Let $\phi: \widehat{KQ} \to \Gamma$ be an algebra homomorphism and 
$R$ be a finite set of basic elements in $J_{\widehat{KQ}}$. 
Then the following conditions are equivalent.

\begin{itemize}
\item[(1)]$\phi$ is surjective and $\Ker\phi=\overline{I}$ for the ideal $I=\langle R\rangle$ of ${\widehat{KQ}}$. 

\item[(2)]
The following sequence is exact for any $i\in Q_0$.
\[\xymatrix@C40pt{  {\displaystyle\bigoplus_{\begin{smallmatrix}r \in R(-,i) \end{smallmatrix}}}\Gamma(\phi s(r)) \ar[r]^{{}_r(\phi\partial^r_ar)_{a}}  & 
{\displaystyle\bigoplus_{\begin{smallmatrix}a \in Q(-,i)\end{smallmatrix}}} \Gamma(\phi s(a)) \ar[r]^{\ \  \ \  \ {}_a(\phi a)} &  J_\Gamma(\displaystyle{\phi i}) \ar[r]& 0.  }\]
\end{itemize}
\end{theorem}

%As usual, we call the above sequence \emph{projective presentation of $J_\Gamma(\displaystyle{\phi i})$}.
%Note that, if $KQ/\langle\mathbf{S}\rangle$ is finite dimensional, we can deal with path algebras $KQ$ instead of $\widehat{KQ}$ and the corresponding statement is true.

First, we give the following observation to obtain some sequences, which play an  important role for the proof. 
We keep the assumption and the notation of Theorem \ref{main}.

\begin{observ}\label{Observ}
For the isomorphism $\phi:\widehat{KQ}/\overline{{\langle R\rangle}}\to\Lambda$, we simply denote $\phi(p)$ by $p$ for any morphism $p$ in $\widehat{KQ}$. 
We have the following complex
\[\xymatrix@C40pt{ 0 \ar[r]^{}  & P_k \ar[r]^{\begin{array}{ccc}\aaa \ \ \ \ \ \end{array}} &{\displaystyle\bigoplus_{\begin{smallmatrix}a\in Q(k,-)\end{smallmatrix}}}P_{e(a)}  \ar[r]^{ \begin{array}{ccc}\partial^l\rr\end{array}} & {\displaystyle\bigoplus_{\begin{smallmatrix}r\in R(k,-)\end{smallmatrix}}}P_{e(r)},     }\]
where we denote by $\aaa:=(a)_a$ and $\partial^l\rr:={}_{a}(\partial_{a}^lr)_{r}$. 

On the other hand, since $P_k$ is an indecomposable projective non-injective module, there exists the following almost split sequence
\[\tag{$i$}\label{al seq}
\xymatrix@C40pt@R20pt{ 0 \ar[r]^{}  & P_k \ar[r]^{\begin{array}{ccc}\aaa \ \ \ \ \ \ \end{array}}  &{\displaystyle\bigoplus_{\begin{smallmatrix}a \in Q(k,-)\end{smallmatrix}}}P_{e(a)}  \ar[r]^{\ \ \begin{array}{ccc}\ \ \g \end{array}  } & \tau^-P_k \ar[r]&0,     }
\] where $\g:={}_a(g_a)$ is defined by the above sequence.
 
Furthermore, for any simple $\Lambda$-module $S_i$ with $i\in Q_0$, we obtain a minimal projective resolution of $S_i$ 
\[\tag{$ii$}\label{proj resol}
\xymatrix@C30pt@R20pt{
\ar[r]& Q \ar[r]^{\begin{array}{ccc} \LL \ \ \ \ \ \  \end{array}} & {\displaystyle\bigoplus_{\begin{smallmatrix}r \in R(-,i) \end{smallmatrix}}} P_{s(r)} \ar[r]^{\begin{array}{ccc} \partial^r\rr \end{array}}  &\ar[r]^{
\begin{array}{ccc} \ \ \ \ \ \bb    \end{array}}  {\displaystyle\bigoplus_{\begin{smallmatrix}b \in Q(-,i) \end{smallmatrix}}} P_{s(b)}  & P_{i}  \ar[r]&S_i\ar[r]&0, \ \ \ }
\]
where we denote by $\bb:={}_b(b)$, $\partial^r\rr:={}_r(\partial_b^rr)_b$ and $\LL$ is defined by the above sequence.

If $P_k$ is a direct summand of ${\displaystyle\bigoplus_{\begin{smallmatrix}r \in R(-,i) \end{smallmatrix}}} P_{s(r)}$, 
we can take the canonical injection 
$e:P_k\to {\displaystyle\bigoplus_{\begin{smallmatrix}r \in R(-,i) \end{smallmatrix}}} P_{s(r)}$.
%For $r\in R(k,i)$, we have the canonical injection $e:P_k\to {\displaystyle\bigoplus_{\begin{smallmatrix}r \in R(-,i) \end{smallmatrix}}} P_{s(r)}$.
Then, by (\ref{al seq}), there exists $\partial^l\partial^r\rr:{\displaystyle\bigoplus_{\begin{smallmatrix}a \in Q(k,-)\end{smallmatrix}}}P_{e(a)} \to{\displaystyle\bigoplus_{\begin{smallmatrix}b \in Q(-,i) \end{smallmatrix}}} P_{s(b)}$ such that 
$e(\partial^r\rr)=\aaa(\partial^l\partial^r\rr)$, where we denote $\partial^l\partial^r\rr:={}_{r,a}(\partial_a^l\partial_b^rr)_{b}$. 
It induces a morphism $f_r:\tau^-P_k\to P_i$ such that $(\partial^l\partial^r\rr)\bb=\g f_r$.

\[
\xymatrix@C50pt@R40pt{
Q \ar[r]^{\begin{array}{ccc} \LL\end{array}\ \ \ }  & {\displaystyle\bigoplus_{\begin{smallmatrix}r \in R(-,i)\end{smallmatrix}}} P_{s(r)} \ar[r]^{\begin{array}{ccc}\partial^r\rr \end{array}}  &\ar[r]^{
\begin{array}{ccc}\ \ \bb \end{array}}  {\displaystyle\bigoplus_{\begin{smallmatrix}b \in Q(-,i) \end{smallmatrix}}} P_{s(b)}
    & P_{i}  \\
0\ar[r] & P_{k} \ar[r]^{\begin{array}{ccc} \aaa\ \  \end{array}} \ar@{^{(}->}[u]^{\begin{array}{ccc}\ \ \ \ e \end{array}} & {\displaystyle\bigoplus_{\begin{smallmatrix}a \in Q(k,-)\end{smallmatrix}}} P_{e(a)} \ar[r]^{\begin{array}{ccc}\ \ \ \g\end{array}} \ar@{.>}[u]^{\begin{array}{ccc}\partial^l\partial^r\rr \end{array}}&  \tau^-P_{k} \ar@{.>}[u]^{\begin{array}{ccc}f_r \end{array}} \ar[r]&0.   }
\] 
Then, by taking a copy of (\ref{al seq}) for every $r\in R(k,i)$, 
we have the following commutative diagram
\[
\xymatrix@C50pt@R40pt{
Q \ar[r]^{\begin{array}{ccc} \LL\end{array}\ \ \ } & {\displaystyle\bigoplus_{\begin{smallmatrix}r \in R(-,i)\end{smallmatrix}}} P_{s(r)} \ar[r]^{\begin{array}{ccc}\partial^r\rr \end{array}}  &\ar[r]^{
\begin{array}{ccc}\ \ \bb \end{array}}  {\displaystyle\bigoplus_{\begin{smallmatrix}b \in Q(-,i) \end{smallmatrix}}} P_{s(b)}
    & P_{i}  \\
0 \ar[r] & {\displaystyle\bigoplus_{\begin{smallmatrix}r \in R(k,i) \end{smallmatrix}}} P_{k} \ar[r]^{\begin{array}{ccc}\oplus \aaa\ \  \end{array}} \ar@{^{(}->}[u]^{\begin{array}{ccc}\ \ \ \ \mathbf{e} \end{array}} & {\displaystyle\bigoplus_{\begin{smallmatrix}a \in Q(k,-), r\in R(k,i) \end{smallmatrix}}} P_{e(a)} \ar[r]^{\begin{array}{ccc}\ \ \  \oplus\g\end{array}} \ar[u]^{\begin{array}{ccc}\partial^l\partial^r\rr \end{array}}&  
{\displaystyle\bigoplus_{\begin{smallmatrix}r\in R(k,i) \end{smallmatrix}}} \tau^-P_{k} \ar[u]^{\begin{array}{ccc}\hh \end{array}} \ar[r]&0,   }
\] 
where we denote by $\mathbf{e}$ the canonical inclusion and   
$\oplus\aaa :=\left(\begin{smallmatrix}\aaa&&0\\ &\ddots&  \\0&&\aaa\end{smallmatrix}\right)$, 
$ \oplus\g :=\left(\begin{smallmatrix} \g&&0\\ &\ddots&  \\0&& \g\end{smallmatrix}\right)$ and $\hh:={}_r(f_r)$. %is defined by the above commutative diagram.
\end{observ}

Next, we give the following lemmas. 
%We keep the notation of Theorem \ref{main} and Observation \ref{Observ}. 
In the rest of this section, we put $T:=T_k$ for simplicity.

\begin{lemm}\label{lift}
Let $h:T\to {\displaystyle\bigoplus_{\begin{smallmatrix}b \in Q(-,i) \end{smallmatrix}}} P_{s(b)}$ be a morphism.
If $h\bb=0$, then there exists $j:T\to{\displaystyle\bigoplus_{\begin{smallmatrix}r \in R(-,i)\end{smallmatrix}}} P_{s(r)}$ such that $j(\partial^r\rr)=h$.
\end{lemm}

\begin{proof}
By the sequence (\ref{proj resol}), we have the following exact sequence
$$0\to\Im \LL\to{\displaystyle\bigoplus_{\begin{smallmatrix}r \in R(-,i)\end{smallmatrix}}} P_{s(r)}\to\Im 
(\partial^r\rr)\to0.$$
Applying the functor $\Hom_\Lambda(T,-)$ to the sequence, 
we have the following exact sequence 
$$\Hom_\Lambda(T,{\displaystyle\bigoplus_{\begin{smallmatrix}r \in R(-,i)\end{smallmatrix}}} P_{s(r)})  \longrightarrow  \Hom_\Lambda(T, \Im(\partial^r\rr)) \longrightarrow \Ext_\Lambda^1(T,\Im \LL).$$
Hence, it is enough to show that $\Ext_\Lambda^1(T,\Im \LL)=0.$

Note that we have $\Ext_\Lambda^3(S_i,P_k)\cong\Hom_\Lambda(Q,P_k)$ since $P_k$ is a simple module. 
Then, by $\id P_k\leq 2$, we have $\Hom_\Lambda(Q,P_k)=0$. 
Therefore, $P_k$ is not a direct summand of $Q$. 
Since $P_k$ is a simple projective module, we have $\Hom_\Lambda(X,P_k)=0$ for any $X\in\mod\Lambda$ if $X\ncong P_k$. 
Therefore, if $P_k$ is a direct summand of $\Im \LL$, then it contradicts the fact that $Q\to\Im \LL$ is an epimorphism.
Thus, $P_k$ is not a direct summand of $\Im \LL.$ 
Hence, by the AR duality, we have $\Ext_\Lambda^1(T,\Im \LL)\cong D\overline{\Hom}_\Lambda(\Im \LL,P_k)=0$. %since $P_k$ is a simple projective module. 
\end{proof}

\begin{lemm}\label{exact}
The following sequence is exact.

\[\tag{$iii$}\label{exact seq}\xymatrix@C80pt{ {\displaystyle
\begin{array}{rr}
{\displaystyle\bigoplus_{\begin{smallmatrix}r \in R(-,i), r\notin R(k,-)
 \end{smallmatrix}}} P_{s(r)}\     \\
{\displaystyle\oplus\bigoplus_{\begin{smallmatrix}a \in Q(k,-),r\in R(k,i) \end{smallmatrix}} P_{e(a)}}
\end{array}}
   \ar[r]^{\ \ \ \ \left(
\begin{array}{ccc}
 \partial^r\rr & 0  \\
\partial^l\partial^r\rr & {\oplus\g} \end{array}
\right)
 } &{\displaystyle
\begin{array}{rr}
{\displaystyle\bigoplus_{\begin{smallmatrix}b \in Q(-,i) \end{smallmatrix}}} P_{s(b)}\ \ \\
{\displaystyle\oplus\bigoplus_{\begin{smallmatrix}r\in R(k,i)\end{smallmatrix}} \tau^-P_{k}}
\end{array}} 
 \ar[r]^{\ \ \ \ \  \ \ \left(\begin{array}{ccc} \bb  \\ -\hh
\end{array}
\right)
} & P_i.     }\]
\end{lemm}

\begin{proof}

By taking the mapping cone of the last commutative diagram of Observation \ref{Observ} and canceling a direct summand ${\displaystyle\bigoplus_{\begin{smallmatrix}r \in R(k,i) \end{smallmatrix}}} P_{k}$, we obtain the desired sequence. 

\end{proof}

\begin{lemm}\label{radical}
Assume that $\begin{pmatrix}p_2& p_1\end{pmatrix}\in\Hom_\Lambda(T,{\displaystyle\bigoplus_{\begin{smallmatrix}b \in Q(-,i) \end{smallmatrix}}} P_{s(b)}\ \ \oplus {\displaystyle\bigoplus_{\begin{smallmatrix}r\in R(k,i)\end{smallmatrix}}} \tau^-P_k )$ satisfies $\begin{pmatrix}p_2& p_1\end{pmatrix}\begin{pmatrix} \bb  \\-\hh \end{pmatrix}=0$.
Then we have $p_1\in \rad_\Lambda(T,{\displaystyle\bigoplus_{\begin{smallmatrix}r\in R(k,i)\end{smallmatrix}}} \tau^-P_k )$.
\end{lemm}

\begin{proof}

%First we will show that $p_1\in \rad_\Lambda(T,{\displaystyle\bigoplus_{\begin{smallmatrix}r\in R(k,i)\end{smallmatrix}}} \tau^-P_k )$. 
Let $p_{1k}:\tau^-P_k\to{\displaystyle\bigoplus_{\begin{smallmatrix}r\in R(k,i)
\end{smallmatrix}}} \tau^-P_k$ be a restriction of $p_1$ 
and $p_{2k}:\tau^-P_k\to{\displaystyle\bigoplus_{\begin{smallmatrix}b\in Q(-,i)
\end{smallmatrix}}} P_{s(b)}$ be a restriction of $p_2$. 
Assume that $p_{1k}$ is a split monomorphism.  
Then there exist split monomorphisms $e:P_k\to {\displaystyle\bigoplus_{\begin{smallmatrix}r \in R(k,i)\end{smallmatrix}}} P_k$ and $e':{\displaystyle\bigoplus_{\begin{smallmatrix}a \in Q(k,-) \end{smallmatrix}}} P_{e(a)}\to{\displaystyle\bigoplus_{\begin{smallmatrix}a \in Q(k,-),r\in R(k,i) 
\end{smallmatrix}}} P_{e(a)}$, which make the diagram below commutative. 
Since $p_{2k}\bb- p_{1k}\hh=0$, 
we have 
$(e'(\partial^l\partial^r\rr)+\g p_{2k})\bb= 
e'(\partial^l\partial^r\rr)\bb-\g p_{1k}\hh=0.$
It implies that 
there exists $v:{\displaystyle\bigoplus_{\begin{smallmatrix}a \in Q(k,-)
\end{smallmatrix}}} P_{e(a)}\to{\displaystyle\bigoplus_{\begin{smallmatrix}r \in R(-,i)
\end{smallmatrix}}} P_{s(r)}$ such that 
$v(\partial^r\rr)=e'(\partial^l\partial^r\rr)+\g p_{2k}.$

Furthermore, since we have $(e\mathbf{e}-\aaa v)(\partial^r\rr)=
e\mathbf{e}(\partial^r\rr)-\aaa e'(\partial^l\partial^r\rr)=0,$
there exists $u:P_k\to Q$ such that $u\LL=e\mathbf{e}-\aaa v.$
Because $e\mathbf{e}$ is a split monomorphism, 
this contradicts the fact that $u\LL+\aaa v$ is not a split monomorphism.

\[\xymatrix@C50pt@R30pt{
Q \ar[r]^{\begin{array}{ccc}\LL\ \ \  \end{array}} & {\displaystyle\bigoplus_{\begin{smallmatrix}r \in R(-,i)\\
\end{smallmatrix}}} P_{s(r)} \ar[r]^{\begin{array}{ccc}\partial^r\rr\ \ \  \end{array}}  &\ar[r]^{
\begin{array}{ccc}\ \ \ \bb \end{array}}  {\displaystyle\bigoplus_{\begin{smallmatrix}b \in Q(-,i) \end{smallmatrix}}} P_{s(b)}
    & P_{i}  \\
0 \ar[r] & {\displaystyle\bigoplus_{\begin{smallmatrix}r \in R(k,i)\end{smallmatrix}}} P_{k} \ar[r]^{\begin{array}{ccc}\oplus \aaa \ \ \  \end{array}} \ar@{^{(}->}[u]^{\begin{array}{ccc}\ \ \ \ \mathbf{e} \end{array}} & {\displaystyle\bigoplus_{\begin{smallmatrix}a \in Q(k,-),r\in R(k,i)
 \end{smallmatrix}}} P_{e(a)} \ar[r]^{\begin{array}{ccc}\ \ \ \ \oplus\g \end{array}} \ar[u]^{\begin{array}{ccc}\partial^l\partial^r\rr \end{array}}&  
{\displaystyle\bigoplus_{\begin{smallmatrix}r\in R(k,i)
\end{smallmatrix}}} \tau^-P_k \ar[u]^{\begin{array}{ccc}\hh \end{array}} \ar[r]&0  \\
0 \ar[r] &  P_{k} \ar@/^8mm/[uul]|{\begin{array}{ccc} u  \end{array}} \ar[r]^{\begin{array}{ccc}\aaa \ \  \end{array}} \ar@{^{(}->}[u]^{\begin{array}{ccc} e \end{array}} & {\displaystyle\bigoplus_{\begin{smallmatrix}a \in Q(k,-)
 \end{smallmatrix}}} P_{e(a)} \ar[r]^{\begin{array}{ccc}\ \ \ \  \g \end{array}} \ar[u]^{\begin{array}{ccc}e' \end{array}} \ar@/^8mm/[uul]|{\begin{array}{ccc} v  \end{array}} &  
 \tau^-P_k \ar[u]^{\begin{array}{ccc}p_{1k} \end{array}} \ar@/^5mm/[uul]|{\begin{array}{ccc} p_{2k}  \end{array}} \ar[r]&0.   }
\] 
\end{proof}

Now, using Theorem \ref{birs2}, we will show that $\PP(\widetilde{\mu}_k^L(Q_\Lambda,W_\Lambda,C_\Lambda))\cong \End_\Lambda(T)$. 
For the associated QP $(Q_\Lambda,W_\Lambda,C_\Lambda)$, we put $Q' :=\widetilde{\mu}_k(Q_\Lambda)$ and define an algebra homomorphism  $\phi':\widehat{KQ'} \to \End_\Lambda(T)$ as follows.   
\begin{itemize}

\item[(1)]Define $\phi' k=\mathbf{pi}\in\End_\Lambda(T)$, where $\mathbf{p}$ is the canonical projection $\mathbf{p}:T\to\tau^-P_k$ and $\mathbf{i}$ is the canonical injection $\mathbf{i}:\tau^-P_k\to T$.
For $i\in Q_0'$ and $i\neq k$, define $\phi'i=\phi i$. 

\item[(2)]For $a\in Q'\cap Q$, define $\phi' a =\phi a$.

\item[(3)]For $a^*,\rho_r^*, [\rho_ra]\in Q_1'$ with $\{a\in Q_1\ |\ s(a)=k\}$ and $\{r\in R\ |\ s(r)=k\}$, define 
$$\phi'(a^*)= g_a \in \Hom_\Lambda(P_{e(a)},\tau^-P_k),$$
$$\phi'(\rho_r^*)=-f_r \in \Hom_\Lambda(\tau^-P_k,P_{e(r)}),$$ 
$$\phi'([\rho_ra])=0,$$
where $g_a$ and $f_r$ are given in Observation \ref{Observ}. 
%(We simply denote $\phi' p$ by $p$ for any morphism $p$ in $KQ'$.)
\end{itemize}

Then, by Theorem \ref{birs2}, it is enough to show the following two lemmas. 
\begin{lemm}\label{lemm1}
%By applying the functor $\Hom_\Lambda(T,-)$ to the sequence $($\ref{al seq}$)$, 
We have the following projective resolution
\[
\xymatrix@C40pt@R20pt{ 0 \ar[r]^{}  &\Hom_\Lambda(T,{\displaystyle\bigoplus_{\begin{smallmatrix}a \in Q(k,-)\end{smallmatrix}}}P_{e(a)})  \ar[r]^{\ \ \begin{array}{ccc}\ \ \cdot\ \g \end{array}  } & \rad_\Lambda(T,\tau^-P_k) \ar[r]&0.     }
\]
%$$0 \longrightarrow \Hom_\Lambda(T,{\displaystyle\bigoplus_{\begin{smallmatrix}a \in Q(k,-)\end{smallmatrix}}}P_{e(a)})  \longrightarrow^{\g\ \ \ \ \ }  \rad_\Lambda(T, \tau^-P_k) \longrightarrow 0.$$ 
\end{lemm}

\begin{proof}
Because (\ref{al seq}) is an almost split sequence and $P_k$ is a simple projective module, we obtain an exact sequence
$$0=\Hom_\Lambda(T,P_k) \longrightarrow \Hom_\Lambda(T,{\displaystyle\bigoplus_{\begin{smallmatrix}a \in Q(k,-)\end{smallmatrix}}}P_{e(a)})  \longrightarrow  \rad_\Lambda(T, \tau^-P_k) \longrightarrow 0.$$ 
\end{proof}

\begin{lemm}\label{lemm2}
%By applying the functor $\Hom_\Lambda(T,-)$ to the sequence $($\ref{exact seq}$)$, we have  a projective presentation of $J_{\End_\Lambda(T)}\phi'(i)$ for $i\in Q_0'$ with $i \neq k$.

For the sequence $($\ref{exact seq}$)$, we put $f_1=\begin{pmatrix} \bb  \\- \hh\end{pmatrix}$ and $f_2=\begin{pmatrix}\partial^r\rr & 0  \\ \partial^l\partial^r\rr &  \oplus \g\end{pmatrix}$. Then 
we have the following projective resolution for $i\in Q_0'$ with $i \neq k$. 
\[\xymatrix@C20pt@R25pt{ \Hom_\Lambda(T,{\displaystyle
\begin{array}{rr}
{\displaystyle\bigoplus_{\begin{smallmatrix}r \in R(-,i), r\notin R(k,-)
 \end{smallmatrix}}} P_{s(r)}\     \\
{\displaystyle\oplus\bigoplus_{\begin{smallmatrix}a \in Q(k,-),r\in R(k,i) \end{smallmatrix}} P_{e(a)}}
\end{array}})
   \ar[r]^{\ \ \ \ \ \cdot \ f_2} &
\Hom_\Lambda(T, {\displaystyle
\begin{array}{rr}
{\displaystyle\bigoplus_{\begin{smallmatrix}b \in Q(-,i) \end{smallmatrix}}} P_{s(b)}\ \ \\
{\displaystyle\oplus\bigoplus_{\begin{smallmatrix}r\in R(k,i)\end{smallmatrix}} \tau^-P_{k}}
\end{array}})  \ar[r]^{\ \ \ \ \ \ \ \  \ \ \cdot \ f_1} & \rad_\Lambda(T,P_i)\ar[r]&0.&     }\]
\end{lemm}

\begin{proof}
%For the sake of simplicity, we put $f_1=\begin{pmatrix} \bb  \\- \hh\end{pmatrix}$ and $f_2=\begin{pmatrix}\partial^r\rr & 0  \\ \partial^l\partial^r\rr &  \oplus \g\end{pmatrix}$.

\emph{Step 1.} We will show that $f_1$ is right almost split in $\add T$.

First, we will show that any morphism $p\in \rad_\Lambda(\Lambda/P_k,P_i)$ factors through $f_1$.
Since $\Lambda/P_k$ is a projective module, 
%there exists $m:\Lambda/P_k \to {\displaystyle\bigoplus_{\begin{smallmatrix}b \in Q(-,i)  \end{smallmatrix}}} P_{s(b)}$ such that $m\bb=p$ by (\ref{proj resol}).
there exists $m:\Lambda/P_k \to {\displaystyle\bigoplus_{\begin{smallmatrix}b \in Q(-,i)  \end{smallmatrix}}} P_{s(b)}\oplus
{\displaystyle\bigoplus_{\begin{smallmatrix}r\in R(k,i)\end{smallmatrix}}} \tau^-P_k$ such that $mf_1=p$ by (\ref{proj resol}). 

Next, we take any morphism $p\in \rad_\Lambda(\tau^-P_k,P_i).$ 
Since ${\displaystyle\bigoplus_{\begin{smallmatrix}a \in Q(k,-) 
 \end{smallmatrix}}} P_{e(a)}$ is a projective module, there exists 
$w_1:{\displaystyle\bigoplus_{\begin{smallmatrix}a \in Q(k,-)  \end{smallmatrix}}} P_{e(a)}\to{\displaystyle\bigoplus_{\begin{smallmatrix}b \in Q(-,i) \end{smallmatrix}}} P_{s(b)}\oplus
{\displaystyle\bigoplus_{\begin{smallmatrix}r\in R(k,i)\end{smallmatrix}}} \tau^-P_k  $ such that 
$w_1f_1=\g p$ by (\ref{proj resol}). 
On the other hand, since the given sequence (\ref{exact seq}) is exact and $\aaa w_1f_1=0$, 
there exists $w_2:P_k\to {\displaystyle\bigoplus_{\begin{smallmatrix}r \in R(-,i),r\notin R(k,-) 
\end{smallmatrix}}} P_{s(r)} \oplus {\displaystyle\bigoplus_{\begin{smallmatrix}a \in Q(k,-),r\in R(k,i) \end{smallmatrix}}} P_{e(a)}$ such that 
$w_2f_2=\aaa w_1$. 
Note that $P_k$ is not a direct summand of ${\displaystyle\bigoplus_{\begin{smallmatrix}r \in R(-,i),r\notin R(k,-) \end{smallmatrix}}} P_{s(r)} \oplus {\displaystyle\bigoplus_{\begin{smallmatrix}a \in Q(k,-),r\in R(k,i) \end{smallmatrix}}} P_{e(a)}$ and, in particular, $w_2$ is not a split monomorphism.  
Then because $\aaa:P_k\to{\displaystyle\bigoplus_{\begin{smallmatrix}a \in Q(k,-) \end{smallmatrix}}} P_{e(a)}$ is left minimal almost split by (\ref{al seq}), there exists $u_1:{\displaystyle\bigoplus_{\begin{smallmatrix}a \in Q(k,-) \end{smallmatrix}}} P_{e(a)}\to{\displaystyle\bigoplus_{\begin{smallmatrix}r \in R(-,i),r\notin 
R(k,-) \end{smallmatrix}}} P_{s(r)} \oplus {\displaystyle\bigoplus_{\begin{smallmatrix}a \in Q(k,-),r\in R(k,i) \end{smallmatrix}}} P_{e(a)}$ such that $w_2=\aaa u_1$. 
Then by $\aaa (w_1-u_1f_2)=0$, there exists $u_2:\tau^-P_k\to {\displaystyle\bigoplus_{\begin{smallmatrix}b \in Q(-,i) \end{smallmatrix}}} P_{s(b)}\oplus
{\displaystyle\bigoplus_{\begin{smallmatrix}r\in R(k,i)\end{smallmatrix}}} \tau^-P_k$ such that 
$\g u_2=w_1-u_1f_2$.
Hence we have $\g u_2f_1=w_1f_1=\g p$ and therefore we obtain $u_2f_1=p$ since 
$\g$ is an epimorphism. 
\[\xymatrix@C60pt@R48pt{
   {\displaystyle
\begin{array}{rr}
{\displaystyle\bigoplus_{\begin{smallmatrix}r \in R(-,i), r\notin R(k,-)
 \end{smallmatrix}}} P_{s(r)}  \ \  \\
{\displaystyle\oplus\bigoplus_{\begin{smallmatrix}a \in Q(k,-),r\in R(k,i) \end{smallmatrix}} P_{e(a)}}
\end{array}} \ar[r]^{\begin{array}{ccc}\ \  f_2 \end{array}} & {\displaystyle
\begin{array}{rr}
{\displaystyle\bigoplus_{\begin{smallmatrix}b \in Q(-,i) \end{smallmatrix}}} P_{s(b)}\ \ \\
{\displaystyle\oplus\bigoplus_{\begin{smallmatrix}r\in R(k,i)\end{smallmatrix}} \tau^-P_{k}}
\end{array}} \ar[r]^{\begin{array}{ccc}\ \ \ \ f_1   \end{array}}  
& P_{i}  \\
P_{k} \ar[r]^{\begin{array}{ccc}\aaa \end{array}} \ar[u]|(0.4){\begin{array}{ccc}  w_2  \end{array}} & {\displaystyle\bigoplus_{\begin{smallmatrix}a \in Q(k,-) \end{smallmatrix}}} P_{e(a)}  \ar[ul]|(0.5){\begin{array}{ccc}u_1 \end{array}} \ar[r]^{\begin{array}{ccc}\ \ \ \  \ \g   \end{array}} \ar[u]|(0.4){\begin{array}{ccc}w_1 \end{array}}&  
 \tau^-P_k \ar[u]|{\begin{array}{ccc}p \end{array}}\ar[ul]|(0.5){\begin{array}{ccc}u_2 \end{array}}  \ar[r]&0.   }
\]

\emph{Step 2.} We will show that $f_2$ is a pseudo-kernel of $f_1$ in $\add T$.

Assume $\begin{pmatrix}p_2& p_1\end{pmatrix}\in\Hom_\Lambda(T,{\displaystyle\bigoplus_{\begin{smallmatrix}b \in Q(-,i) \end{smallmatrix}}} P_{s(b)}\oplus {\displaystyle\bigoplus_{\begin{smallmatrix}r\in R(k,i)\end{smallmatrix}}} \tau^-P_k )$ satisfies $\begin{pmatrix}p_2& p_1\end{pmatrix}\begin{pmatrix} \bb  \\-\hh \end{pmatrix}=0$.
By Lemma \ref{radical}, we can assume that $p_1\in \rad_\Lambda(T,{\displaystyle\bigoplus_{\begin{smallmatrix}r\in R(k,i)\end{smallmatrix}}} \tau^-P_k )$.
Since $\g$ is right minimal almost split by (\ref{al seq}), there exists $q_1\in\Hom_\Lambda(T,{\displaystyle\bigoplus_{\begin{smallmatrix}a \in Q(k,-), r\in R(k,i) \end{smallmatrix}}} P_{e(a)})$ such that 
$p_1= q_1(\oplus \g)$.
Hence, we have 
%$$(p_2-q_1(\partial^l\partial^r\rr))\bb=p_2\bb-q_1(\partial^l\partial^r\rr)\bb=p_2\bb- q_1(\oplus \g)\hh=p_2\bb- p_1\hh=0.$$
\begin{eqnarray*}
(p_2-q_1(\partial^l\partial^r\rr))\bb&=&p_2\bb-q_1(\partial^l\partial^r\rr)\bb \\
&=&p_2\bb- q_1(\oplus \g)\hh\\
&=&p_2\bb- p_1\hh \\
&=&0.
\end{eqnarray*}

Then, by Lemma \ref{lift}, 
there exists $q_2 :T \to {\displaystyle\bigoplus_{\begin{smallmatrix}r \in R(-,i) \end{smallmatrix}}} P_{s(r)}$ such that 
$q_2(\partial^r\rr) = p_2-q_1(\partial^l\partial^r\rr)$. 
Therefore, we have $\left(\begin{matrix}q_2&q_1\end{matrix}\right)
\left(\begin{smallmatrix}\partial^r\rr&0\\   \\\partial^l\partial^r\rr&\oplus\g\end{smallmatrix}\right)=
\left(\begin{matrix}p_2&p_1\end{matrix}\right).$

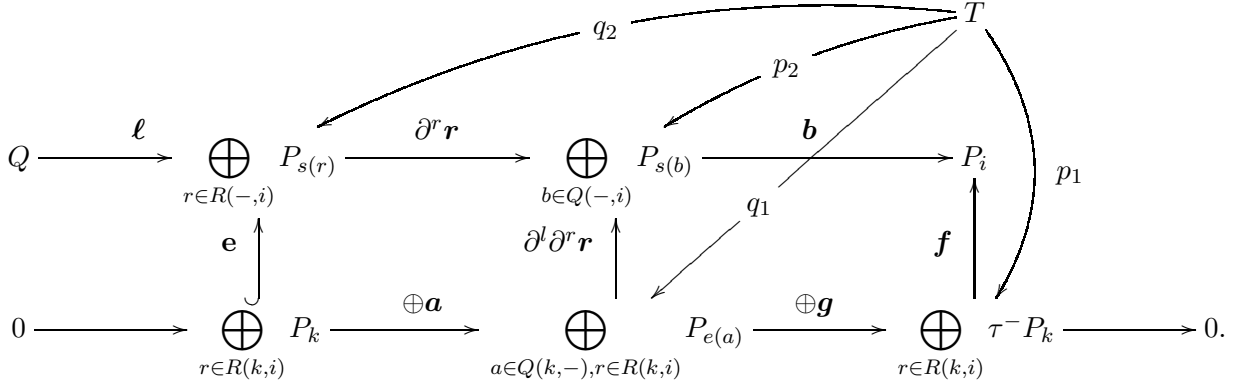
\[
\xymatrix@C50pt@R30pt{
&&&T \ar@/^8mm/[dd]^{\begin{array}{ccc}p_1 \end{array}}  \ar@/_3mm/[dl]|{\begin{array}{ccc}p_2 \end{array}} \ar[ddl]|(0.6){\begin{array}{ccc}q_1 \end{array}}  \ar@/_8mm/[dll]|{\begin{array}{ccc}q_2 \end{array}}    &\\
Q \ar[r]^{\begin{array}{ccc}\LL \end{array}} & {\displaystyle\bigoplus_{\begin{smallmatrix}r \in R(-,i)
\end{smallmatrix}}} P_{s(r)} \ar[r]^{\begin{array}{ccc}\partial^r\rr \end{array}}  &\ar[r]^{
\begin{array}{ccc}\ \ \ \bb \end{array}}  {\displaystyle\bigoplus_{\begin{smallmatrix}b \in Q(-,i) \end{smallmatrix}}} P_{s(b)}  & P_{i}  \\
0 \ar[r] & {\displaystyle\bigoplus_{\begin{smallmatrix}r \in R(k,i)\end{smallmatrix}}} 
P_{k} \ar[r]^{\begin{array}{ccc}\oplus \aaa \ \ \  \end{array}} \ar@{^{(}->}[u]^{\begin{array}{ccc}\ \ \ \ \mathbf{e} \end{array}} & {\displaystyle\bigoplus_{\begin{smallmatrix}a \in Q(k,-),r\in R(k,i) \end{smallmatrix}}} P_{e(a)} \ar[r]^{\begin{array}{ccc}\ \ \ \ \oplus \g \end{array}} \ar[u]^{\begin{array}{ccc}\partial^l\partial^r\rr \end{array}}&  
{\displaystyle\bigoplus_{\begin{smallmatrix}r\in R(k,i)
\end{smallmatrix}}} \tau^-P_k \ar[u]^{\begin{array}{ccc}\hh \end{array}} \ar[r]&0.   }
\] 
\end{proof}
Thus, we have completed the proof of Theorem \ref{main}.

\begin{remk}
We remark that we cannot drop the assumption $\id P_k\leq 2$. 
For example, let $\Lambda$ be the algebra given by following quiver with relations. 
$$
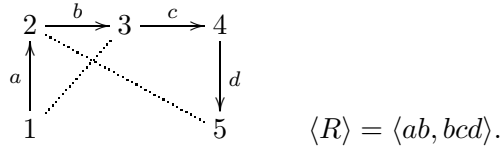
\xymatrix{
2\ar[r]^b &  3\ar[r]^c \ar@{.}[ld]&4 \ar[d]^d  \\
1 \ar[u]^a& & 5 \ar@{.}[llu] & \langle R\rangle=\langle ab,bcd\rangle.} $$ 
Then we can calculate $Q'$ and $R'$ satisfying $\widehat{KQ'}/\overline{{\langle R'\rangle}}\cong \End_\Lambda(T_1)$, which is given by the following quiver with relations. 
%Note that these algebras were investigated in \cite[Proposition 2.3]{T}. 
 
$$
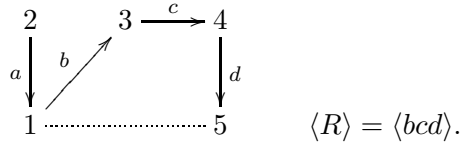
\xymatrix{
2\ar[d]_a &  3\ar[r]^c &4 \ar[d]^d  \\
1 \ar[ru]^b& & 5 \ar@{.}[ll] &\langle R\rangle=\langle bcd\rangle.} $$
%\xymatrix{
%2\ar[d]_a &  3\ar[r]^c &4 \ar[d]^d  \\
%1 \ar[ru]^b& & 5, \ar@{--}[llu] &R=\langle abcd\rangle.}
%$
\end{remk}

\begin{remk}
We remark that we cannot define the APR co-tilting module associated with $k$ in $\End_\Lambda(T)$ since the vertex $k$ is not sink. 
In this case, however, we can define a BB co-tilting $\Gamma$-module $T_k':=\tau S_k\oplus D\Gamma/I_k$ \cite{BB}, where $\Gamma:=\End_\Lambda(T)$, and we have $\End_\Gamma(T_k')\cong\PP(\mu_k^R\mu_k^L(Q_\Lambda,W_\Lambda,C_\Lambda))\cong\Lambda$. 
Note that this is also pointed out in \cite[Proposition 3.2]{L}.

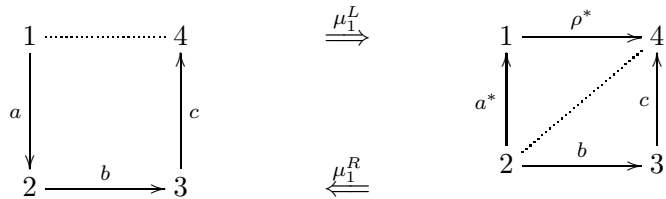
$$\xymatrix@C45pt@R35pt{
1\ar[d]_{a}    &4 \ar@{.}[l]&\overset{\mu_1^L}{\Longrightarrow}  \\
2 \ar[r]^b   &  3 \ar[u]_{c} & \overset{\mu_1^R}{\Longleftarrow} }\ \ \ \ \ \ \  \ \
\xymatrix@C45pt@R35pt{
1  \ar[r]^{\rho^*}   &4\ar@{.}[ld]     \\
2\ar[r]^b \ar[u]^{a^*}   & 3 \ar[u]^c}$$
$$\ \ \ \ \ \ \ \ \ \ \ \ \ \ \ \ \ \ \ \ \ \ \langle R\rangle=\langle abc \rangle.  \ \ \ \ \ \ \ \ \ \ \ \ \ \ \ \ \ \ \ \ \ \ \ \ \ \ \ \ \ \  \ \langle R'\rangle=\langle a^*\rho^*+bc  \rangle.\ \ \ \ \ \ \ \ \ \ \ \ \ \ \ $$
\end{remk}

%%%%%%%%%%%%%%%%%%%%%%%%%%%%%%%%%%%%%%%%%%%%
\section{QPs with algebraic cuts and mutations}
%%%%%%%%%%%%%%%%%%%%%%%%%%%%%%%%%%%%%%%%%%%%
In this section, we formulate the previous result in terms of algebraic cuts and consider successive mutations of QPs. 
We will show that QPs with algebraic cuts are closed under left mutations (respectively, right mutations) at a strict source (strict sink).  
As a consequence, mutations of QPs give APR-tilted (co-tilted) algebras and, in particular, provide 
a rich source of derived equivalent algebras with global dimension at most 2. 

In the rest of this paper, we assume that $Q$ is a finite connected quiver which is not only one vertex for simplicity. 
\begin{defi}\label{strict}\cite{HI}
Let $C$ be a cut of a QP $(Q,W)$.
 
$\bullet$ We say that a vertex $k$ of $Q$ is a \emph{strict source} (respectively, \emph{strict sink})
if all arrows ending (starting) at $k$ belong to $C$ and all arrows starting (ending) at $k$ do not belong to $C$. 
Note that a strict source (respectively, strict sink) in $Q$ gives to a source (sink) of $Q_C$.
%, where $Q_C$ is the subquiver of $Q$ with vertex set $Q_0$ and arrow set $Q_1\setminus C$.

$\bullet$ $C$ is called \emph{algebraic} if the following conditions are satisfied.
\begin{itemize}
\item[(1)] $\PP(Q,W,C)$ is a finite dimensional algebra with global dimension at most two.
\item[(2)] $\{\partial_cW\}_{c \in C}$ is a minimal set of generators of the ideal $\overline{{\langle \partial_c W \;|\; c \in C \rangle}}$ of $\widehat{{KQ_C}}$.
\end{itemize}
We denote by $\QQ_1$ the set of reduced QPs with algebraic cuts. 
%Remark that, if $(Q,W,C) \in \QQ_1$, then $\overline{{\langle \partial_c W \;|\; c \in C \rangle}}$ is a minimal set of relations of $\widehat{KQ_C}$.
\end{defi}

%Remark that algebraic cuts are crucial for the understanding of \emph{2-representation-finite} (\emph{2-representation-infinite}) algebras \cite{I} \cite{AIR} and bimodule 3-CY algebras \cite{G,HI,K,TV}. 
%The link between QPs with cuts and motivic Donaldson-Thomas invariants is explained in \cite{N}.

Then we will restate Theorem \ref{main} in terms of algebraic cuts and tilting mutations. This proposition implies that left mutations induce left tilting mutations of the truncated Jacobian algebra.

\begin{prop}\label{main3}
Let $(Q,W,C)$ be a QP of $\QQ_1$ and $k$ be its strict source. We put $\Lambda:=\PP(Q,W,C)$. Then we have $$\End_\Lambda(\nu_k^L(\Lambda))\cong\PP(\mu_{k}^L(Q,W,C)),$$
where $\nu_k^L$ is a left tilting mutation.
\end{prop}

Then, it is natural to consider iterated mutations of QPs. Indeed, 
we will show that $\mu_{k}^L(Q,W,C)$ also belongs to $\QQ_1$ and therefore we can apply Proposition \ref{main3} repeatedly. 
The following easy observation is useful.

\begin{lemm}\label{reduced}
Let $(Q,W,C)$ be a QP with a cut and assume that 
$(Q,W,C)$ is graded right-equivalent to the direct sum 
$$(Q_{red},W_{red},C_{red})\oplus(Q_{tri},W_{tri},C_{tri}),$$ 
where $(Q_{red},W_{red},C_{red})$ is reduced and $(Q_{tri},W_{tri},C_{tri})$ is trivial. 
If $\{\partial_cW\}_{c \in C}$ is a minimal set of generators of the ideal $\overline{{\langle \partial_c W \;|\; c \in C \rangle}}$ of $\widehat{KQ_{C}}$, then $\{\partial_cW_{red}\}_{c \in C_{red}}$ is also a minimal set of generators of the ideal $\overline{{\langle \partial_c W_{red} \;|\; c \in C_{red} \rangle}}$ of $\widehat{K(Q_{red})_{C_{red}}}.$ 
\end{lemm}

\begin{proof} 
We put $Q':=Q_{red}\oplus Q_{tri}$, $W':=W_{red}+W_{tri}$, and $C':=C_{red}+C_{tri}$. 
Then, by the assumption, $\{\partial_cW'\}_{c \in C'}$ is a minimal set of generators of ${\langle \partial_c W' \;|\; c \in C' \rangle}$. 
Because $\{\partial_cW_{red}\}_{c \in C_{red}}$ and 
$\{\partial_cW_{tri}\}_{c \in C_{tri}}$ are linearly independent in ${\langle \partial_c W' \;|\; c \in C' \rangle}$, 
the assertion follows. 
\end{proof} 

Then, we show that $\QQ_1$ are closed under left mutations at a strict source.

\begin{theorem}\label{algebraic}
Let $(Q,W,C)$ be a QP of $\QQ_1$ and $k$ be its strict source.
Then we have $\mu_k^L(Q,W,C)\in \QQ_1$.
\end{theorem}

\begin{proof} 
%First, by the argument of Subsection 2.1, degree 1 arrows give a cut of $\mu_k^L(Q,W,C)$. 

%$\mu_k^L(Q,W,C)$ is homogeneous of degree 1. Therefore%from the assumption of a strict source and by \cite[Proposition 6.5]{AO1} and%and the definition of $\widetilde{\mu}_k^L$, 
%any arrow of $\widetilde{\mu}_k^L(Q,W,C)$ has degree 0 or 1, and hence 
%any arrow of $\mu_k^L(Q,W,C)$ also has degree 0 or 1. On the other hand, 
We put $\Lambda:=\PP(Q,W,C)$. Then for the APR tilting $\Lambda$-module $T_k$, we have $\PP(\mu^L_k(Q,W,C))\cong \End_\Lambda(T_k)$ by Theorem \ref{main}. 
Thus, $\PP(\mu^L_k(Q,W,C))$ is a finite dimensional algebra, and 
we have $\gl\PP(\mu^L_k(Q,W,C))\leq 2$ by \cite[Proposition 1.15]{APR}. 

Next, we put $(Q',W',C'):=\widetilde{\mu}^L_k(Q,W,C)$. 
By the assumption of $(Q,W,C)$, $\overline{{\langle \partial_{c} W \;|\; c \in C \rangle}}$ is a minimal set of relations of $\widehat{KQ_{C}}$ and $\gl\Lambda\leq2$. 
Then the projective resolutions of Lemmas \ref{lemm1}, \ref{lemm2} have length at most 1 and hence $\{\partial_{c'}W'\}_{c' \in C'}$ is a minimal set of generators of $\overline{{\langle \partial_{c'} W' \;|\; c' \in C' \rangle}}$ in $\widehat{KQ_{C'}}$. %(cf. \cite[Section 3]{BIRSm}). 
Then, by the graded splitting theorem and Lemma \ref{reduced}, we have $\mu_k^L(Q,W,C)\in \QQ_1$.
\end{proof}
Summarizing the preceding results and their dual statements, we obtain the following corollary.  
In the corollary, we consider the following correspondence.

\begin{itemize}
\item[(1)] If $k$ is a strict source, 
we denote by $\nu_k$ left tilting mutation and by $\mu_k$ left mutation of a QP. 

\item[(2)] If $k$ is a strict sink, 
we denote by $\nu_k$ right co-tilting mutation and by $\mu_k$ right mutation of a QP. 

\end{itemize}
We put ${\upsilon}_k(\Lambda):=\End_\Lambda(\nu_k(\Lambda))$ for simplicity.

\begin{cor}\label{iterated mut} 
Let $(Q,W,C)$ be a QP of $\QQ_1$ and put $\Lambda:=\PP(Q,W,C)$. 
Assume that a vertex $k_i$ is a strict source or strict sink in $\mu_{k_{1}}\circ\cdots\circ\mu_{k_{i-1}}(Q,W,C)$ for any $1\leq i\leq N$.
Then we have 
$$\PP(\mu_{k_1}\circ\cdots\circ\mu_{k_N}(Q,W,C))\cong\upsilon_{k_1}\circ\cdots\circ\upsilon_{k_N}(\Lambda).$$

\end{cor}
\begin{proof} 
By Proposition \ref{main3} or its dual statement, we have an isomorphism 
$$\PP(\mu_{k_1}(Q,W,C))\cong\End_\Lambda(\nu_{k_1}(\Lambda)).$$
Then, by Proposition \ref{algebraic} or its dual statement, we have $\mu_{k_1}(Q,W,C)\in\QQ_1$.
Therefore, we can apply Proposition \ref{main3} repeatedly and the conclusion follows.
\end{proof}

\begin{exam}
(1) It is known that, for a concealed algebra $\Lambda$ of the Euclidean type $\widetilde{Q}$, 
$\Lambda$ may be transformed by a sequence of APR tilts to $K\widetilde{Q}$ \cite{H}. 
Then we can restate the claim by saying that the associated QP of $\Lambda$ may be transformed by a sequence of left mutations at strict sources to $(\widetilde{Q},0,\emptyset)$.   
The left-hand side algebra, where the relations are given by commutative relations, is a concealed algebra $\Lambda$ of the Euclidean type $\widetilde{\mathbb{D}}_9$ \cite{SS} and 
mutation procedures immediately show that we can obtain $K\widetilde{\mathbb{D}}_9$ by a sequence of APR tilts at the vertices 1,2, and 3. 

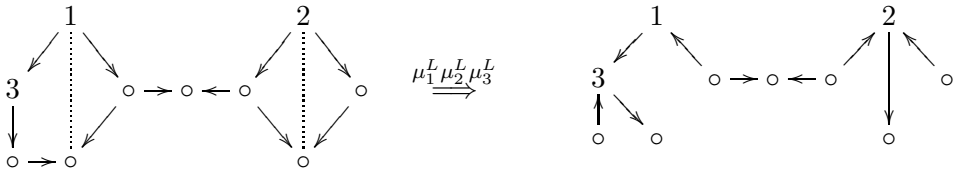
\[\xymatrix@C10pt@R10pt{
  &1 \ar[ld] \ar[rd] \ar@{.}[dd] & &&   &2 \ar[ld] \ar[rd] \ar@{.}[dd]  & \\
3\ar[d]           &  &\circ \ar[ld] \ar[r]&\circ&\circ\ar[rd] \ar[l]  &  &\circ \ar[ld] &\overset{\mu_1^L\mu_2^L\mu_3^L}{\Longrightarrow}\\
\circ \ar[r] &\circ    & & &  &\circ      & }\ \ \ \ \ \ \ \  
\xymatrix@C10pt@R10pt{
  &1 \ar[ld]  & &&   &2   \ar[dd]  & \\
3  \ar[rd]         &  &\circ \ar[lu] \ar[r]&\circ&\circ\ar[ru] \ar[l]  &  &\circ \ar[lu] \\
\circ \ar[u] &\circ    & & &  &\circ      & }\]

(2) 
In \cite{HI}, the authors gave a method to obtain a selfinjective QP (i.e., the Jacobian algebra is finite dimensional and selfinjective) by mutations of QPs from a given selfinjective one. 
In the following diagram, 
from a left-hand side selfinjective QP, we can obtain a new selfinjective one 
by mutations $\mu_1\mu_6\mu_3$.
$$
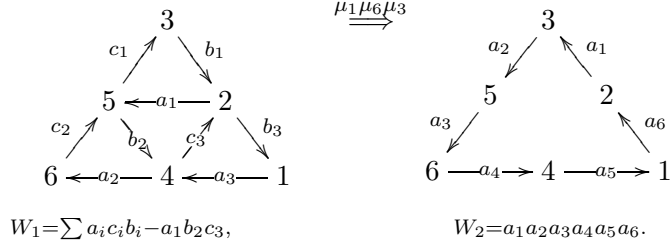
\xymatrix@C10pt@R15pt{
 & & 3 \ar[rd]^{b_1}& && \overset{\mu_1\mu_6\mu_3}{\Longrightarrow}\\
 & 5\ar[ru]^{c_1}  \ar[rd]|{b_2}& & 2\ar[rd]^{b_3} \ar[ll]|{a_1}&& \\
6 \ar[ru]^{c_2}& & 4 \ar[ll]|{a_2} \ar[ru]|{c_3} & & 1\ar[ll]|{a_3} } 
\xymatrix@C10pt@R15pt{
 & & 3 \ar[ld]_{a_2}& & \\
 & 5  \ar[ld]_{a_3} & & 2\ar[lu]_{a_1} & \\
6 \ar[rr]|{a_4}& & 4 \ar[rr]|{a_5}  & & 1\ar[lu]_{a_6}}$$
$${}_{W_1=\sum a_ic_ib_i-a_1b_2c_3, \ \ \ \ \ \ \ \ \ \ \ \ \ \ \ \ \  \ \ \ \ \ \  W_2=a_1a_2a_3a_4a_5 a_6. \ \ \ \ \ \ }$$

Moreover, they showed that, if a selfinjective QP have a cut, a \emph{2-representation-finite algebra} \cite{IO1} was given by taking the truncated Jacobian algebras.    
In this examples we can take, for instance, $\{a_1,a_2,a_3 \}$ as a cut of $W_1$ and $\{a_1 \}$ as a cut of $W_2$ and obtain two 2-representation-finite algebras.  
We are now able to know the relationship between them.  
The following diagram immediately shows that they are derived equivalent.
%(Note that $(\mu_3^L)^{-1}=\mu_3^R$ is corresponding to the BB co-tilting module). 
$$
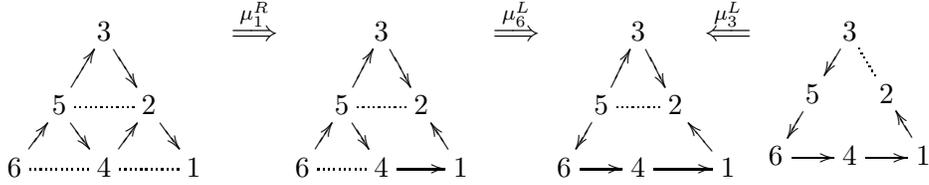
\xymatrix@C5pt@R10pt{
 & & 3 \ar[rd]& &&\overset{\mu_1^R}{\Longrightarrow} \\
 & 5\ar[ru] \ar@{.}[rr] \ar[rd]& & 2\ar[rd] && \\
6 \ar@{.}[rr]\ar[ru]& & 4 \ar@{.}[rr] \ar[ru] & & 1}
\xymatrix@C3pt@R10pt{
 & & 3 \ar[rd]& & &\overset{\mu_6^L}{\Longrightarrow}\\
 & 5\ar[ru] \ar@{.}[rr] \ar[rd]& & 2 & \\
6 \ar@{.}[rr]\ar[ru]& & 4 \ar[rr]  & & 1\ar[lu]}
\xymatrix@C2pt@R10pt{
 & & 3 \ar[rd]& & \overset{\mu_3^L}{\Longleftarrow}\\
 & 5\ar[ru] \ar@{.}[rr] \ar[ld] & & 2 & \\
6 \ar[rr]& & 4 \ar[rr]  & & 1\ar[lu]}
\xymatrix@C2pt@R10pt{
 & & 3 \ar[ld]\ar@{.}[rd]& & \\
 & 5  \ar[ld] & & 2 & \\
6 \ar[rr]& & 4 \ar[rr]  & & 1\ar[lu]}$$
This example also shows that, while mutations of QPs do not correspond to the tilting mutations for the first two Jacobian algebras (since they are selfinjective), they often induce 
tilting mutations for the truncated Jacobian algebras by Proposition \ref{main3}.  
\end{exam}
In this way, we can treat various kinds of representation theoretical arguments by purely combinatorial
methods.

\section{application}

Finally we pick up the following question. 

\begin{question}\cite[Question 12.2]{DWZ}
Let $(Q,W)$ be a QP without loops and $k\in Q$ is a vertex not lying on any 2-cycle.  
If $\PP(Q,W)\cong\PP(Q',W')$ for some QP $(Q',W')$,  
do we have $\PP(\mu_k(Q,W))\cong\PP(\mu_k(Q',W'))$ ? 
\end{question}
As an immediate consequence of previous result, we give a sufficient condition of QPs which have a positive answer.

\begin{theorem}\label{app1}
Let $(Q,W,C)$ and $(Q',W',C')$ be QPs of $\QQ_1$ and $k$ be a strict source or strict sink of $(Q,W,C)$.
If $\PP(Q,W,C) \cong\PP(Q',W',C')$, 
then we have algebra isomorphisms $\PP(Q,W)\cong\PP(Q',W')$ and $\PP(\mu_k(Q,W))\cong\PP(\mu_k(Q',W'))$.  
\end{theorem}

For the proof, we recall the following notion and its property. 

\begin{defi}\cite{IO,K}
Let $\Lambda$ be a finite dimensional algebra of global dimension at most 2. 
The \emph{complete 3-preprojective algebra} is defined as the tensor algebra 
$$\widehat{\Pi}_3(\Lambda):={\displaystyle\prod_{\begin{smallmatrix}i \geq0\end{smallmatrix}}}\Ext_\Lambda^2(D\Lambda,\Lambda)^{\otimes^i_\Lambda}$$
of the $(\Lambda,\Lambda)$-module $\Ext_\Lambda^2(D\Lambda,\Lambda)$.
\end{defi}

\begin{prop}\label{preproj QP2}
Let $\Lambda=\widehat{KQ}/\overline{{\langle R\rangle}}$ be a finite dimensional algebra with a minimal set $R$ of relations and $\gl\Lambda\leq2$. 
Then we have an algebra isomorphism $\widehat{\Pi}_3(\Lambda)\cong\PP(Q_\Lambda,W_\Lambda).$ 
\end{prop}

The above result is known for experts and also follows from the recent result of Keller \cite{K}, which is treating more general situations. %For the convenience of the reader, we give a proof.

This result immediately yields the following result.
\begin{cor}\label{3-preproj}
Let $(Q,W,C)$ and $(Q',W',C')$ be QPs of $\QQ_1$.
If $\PP(Q,W,C) \cong\PP(Q',W',C')$, then we have $\PP(Q,W) \cong\PP(Q',W')$.
\end{cor}

\begin{proof}
Since we have $\PP(Q,W)\cong\widehat{\Pi}_3(\PP(Q,W,C))$ by Proposition \ref{preproj QP2}, 
we obtain  
$$\PP(Q,W)\cong\widehat{\Pi}_3(\PP(Q,W,C))\cong\widehat{\Pi}_3(\PP(Q',W',C')) \cong\PP(Q',W').$$
\end{proof}

Then we give a proof of Theorem \ref{app1}.

\begin{proof}
We assume that $k$ is a strict source. The proof is similar for a strict sink. 
By the assumption and Corollary \ref{3-preproj},  
we have $\PP(Q,W)\cong\PP(Q',W')$. 
On the other hand, by Theorem \ref{main}, $\PP(\mu_k^L(Q,W,C))$ is determined by the algebra $\PP(Q,W,C)$ 
and does not depend on the choice of a QP. 
Therefore, we have $\PP(\mu_k^L(Q,W,C)) \cong\PP(\mu_k^L(Q',W',C'))$. 
Then, by Theorem \ref{algebraic} and Corollary \ref{3-preproj}, we have $\PP(\mu_k(Q,W)) \cong\PP(\mu_k(Q',W')).$
\end{proof}

\subsection*{Funding}
This work was supported by Grant-in-Aid for Japan Society for the Promotion of Science Fellowships No.23.5593.

\subsection*{Acknowledgments}
First and foremost, the author would like to thank O. Iyama for his support and patient guidance.
The author wishes to thank M. Herschend, S. Oppermann and D. Yang for their comments and stimulating discussions.
The author is grateful to K. Yamaura and T. Adachi for their help and advice. 
The author is very grateful to the anonymous referee for his or her carefully reading the manuscript, and for the valuable comments and suggestions, which improve the presentation.
%%%%%%%%%%%%%%%%%

\end{document}